\documentclass[12pt]{article}
\usepackage{amsmath,amssymb,amsfonts}
\begin{document}
\newtheorem{theorem}{Theorem}[section]
\newtheorem{corollary}[theorem]{Corollary}
\newtheorem{lemma}[theorem]{Lemma}
\newtheorem{remark}[theorem]{Remark}
\newtheorem{example}[theorem]{Example}
\newtheorem{proposition}[theorem]{Proposition}
\newtheorem{definition}[theorem]{Definition}
\def\emptyset{\varnothing}
\def\setminus{\smallsetminus}
\def\id{{\mathrm{id}}}
\def\G{{\mathcal{G}}}
\def\H{{\mathcal{H}}}
\def\C{{\mathbb{C}}}
\def\N{{\mathbb{N}}}
\def\Q{{\mathbb{Q}}}
\def\R{{\mathbb{R}}}
\def\Z{{\mathbb{Z}}}
\def\Path{{\mathrm{Path}}}
\def\Str{{\mathrm{Str}}}
\def\st{{\mathrm{st}}}
\def\tr{{\mathrm{tr}}}
\def\opp{{\mathrm{opp}}}
\def\a{{\alpha}}
\def\be{{\beta}}
\def\de{{\delta}}
\def\e{{\varepsilon}}
\def\si{{\sigma}}
\def\la{{\lambda}}
\def\th{{\theta}}
\def\lan{{\langle}}
\def\ran{{\rangle}}
\def\isom{{\cong}}
\newcommand{\Hom}{\mathop{\mathrm{Hom}}\nolimits}
\newcommand{\End}{\mathop{\mathrm{End}}\nolimits}
\def\qed{{\unskip\nobreak\hfil\penalty50
\hskip2em\hbox{}\nobreak\hfil$\square$
\parfillskip=0pt \finalhyphendemerits=0\par}\medskip}
\def\proof{\trivlist \item[\hskip \labelsep{\bf Proof.\ }]}
\def\endproof{\null\hfill\qed\endtrivlist\noindent}

%%%%%%%%%%%%%%%%%%%%%%%%%%%%%%
\title{A characterization of \\
a finite-dimensional commuting square\\
producing a subfactor of finite depth}
\author{
{\sc Yasuyuki Kawahigashi}\\
{\small Graduate School of Mathematical Sciences}\\
{\small The University of Tokyo, Komaba, Tokyo, 153-8914, Japan}\\
{\small e-mail: {\tt yasuyuki@ms.u-tokyo.ac.jp}}
\\[0,40cm]
{\small Kavli IPMU (WPI), the University of Tokyo}\\
{\small 5--1--5 Kashiwanoha, Kashiwa, 277-8583, Japan}
\\[0,40cm]
{\small Trans-scale Quantum Science Institute}\\
{\small The University of Tokyo, Bunkyo-ku, Tokyo 113-0033, Japan}
\\[0,05cm]
{\small and}
\\[0,05cm]
{\small iTHEMS Research Group, RIKEN}\\
{\small 2-1 Hirosawa, Wako, Saitama 351-0198,Japan}}
\maketitle{}

\begin{abstract}
We give a characterization of a finite-dimensional commuting square
of $C^*$-algebras with a normalized trace that produces 
a hyperfinite type II$_1$ subfactor of finite index and finite depth
in terms of Morita equivalent unitary fusion categories.  
This type of commuting squares was studied by N. Sato, and we
show that a slight generalization of
his construction covers the fully general case of such
commuting squares.  We also give a characterization of such
a commuting square that produces a given hyperfinite 
type II$_1$ subfactor of finite index and finite depth.
These results also give a characterization of certain
4-tensors that appear in recent studies of matrix product
operators in 2-dimensional topological order.
\end{abstract}

\section{Introduction}

\textsl{Subfactor} theory of Jones \cite{J} has revealed a wide range of
interconnections among many different topics in mathematics and
physics.  In mathematical studies of subfactors, a finite
dimensional \textsl{commuting square}, a combination of four
finite dimensional $C^*$-algebras with a trace satisfying a certain
compatibility condition, has played important roles
both in construction of examples and classification.  (See the
following section for a precise setting.)  Repeated basic
constructions of Jones applied to a finite dimensional commuting 
square produces a hyperfinite II$_1$ subfactor of finite index
as the limit algebras.
A type II$_1$ subfactor of finite index produces a sequence
of finite dimensional commuting square arising as \textsl{higher relative
commutants} through the Jones tower construction, 
called the \textsl{canonical} commuting squares.
If the original subfactor is \textsl{strongly amenable}
in the sense of \cite{P2},
this sequence recovers the original subfactor completely, due
to the celebrated classification theorem of Popa \cite{P2}.
If the \textsl{Bratteli diagrams} of the higher relative commutants
stabilize after a certain stage,
the original subfactor is said to be of \textsl{finite depth}.  A
hyperfinite II$_1$ subfactor of finite index and finite
depth is automatically strongly amenable.  For this reason,
a finite dimensional, but not necessarily canonical,
commuting square producing a subfactor of finite index
and finite depth has been important for more than 30 years.

The finite depth condition is also related to \textsl{$(2+1)$-dimensional
topological quantum field theory} and 
\textsl{$2$-dimensional conformal field theory}.
A type II$_1$ subfactor of finite index and finite depth
produces a \textsl{fusion category} of bimodules and the
connections to quantum topology and mathematical physics work
out with such categories.
Its relation to \textsl{2-dimensional topological order} has created
recent interest in subfactors of finite index and finite depth.
A finite dimensional
commuting square produces a \textsl{bi-unitary connection}, 
a certain kind of \textsl{$4$-tensor}, and a tensor network plays
the key role in recent studies of 2-dimensional topological order
\cite{BMWSHV}, \cite{K6}, \cite{LFHSV}.  (Also see
\cite{K4}, \cite{LW} for related topics.)  For these reasons,
we are interested in when a finite dimensional commuting square
produces a hyperfinite II$_1$ subfactor of finite index and
finite depth through repeated basic constructions.  
Our main theorem in this paper, Theorem \ref{main},
gives a characterization of such a commuting square.  Sato's results
\cite{S1}, \cite{S2}, \cite{S3} play the key roles in this 
characterization.  As a byproduct, we also give a characterization
of such a commuting square that produces a given fixed subfactor
of finite index and finite depth.  It is easy to see that countably
many different commuting squares produce the same subfactor of
finite index and finite depth, so we never have uniqueness.
We determine which commuting square
produce a specific subfactor.  We then give examples related to
the $A$-$D$-$E$ Dynkin diagrams and the Goodman-de la Harpe-Jones
subfactors.

We list \cite{EK2}, \cite{GHJ} and \cite{JS} for general references
on subfactor theory, \cite{BKLR}, \cite{K3} for connections to conformal
field theory and tensor categories, and \cite{K5} for 
relations to $2$-dimensional topological order.

This work was partially supported by 
JST CREST program JPMJCR18T6 and
Grants-in-Aid for Scientific Research 19H00640
and 19K21832. 

\section{Finite dimensional commuting squares and construction of subfactors}

We review a general construction of a hyperfinite type II$_1$ subfactor
from repeated basic constructions of a commuting square of finite
dimensional $C^*$-algebras after some basic definitions and results.
They are standard materials as in books
\cite[Chapter 9]{EK2}, \cite[Chapter 4]{GHJ},
\cite[Chapter 5]{JS}, but there is a small subtlety about 
so-called non-degeneracy, so we explain the construction in detail
and fix the notation.

Let $A\subset B$ be an inclusion of finite von Neumann algebras with
a faithful normalized trace $\tr$.  (That is, we require $\tr(1)=1$.)
Then the conditional expectation $E_A$ from $B$ to $A$ is uniquely
determined by the property $\tr(ab)=\tr(a E_A(b))$ for
$a\in A$, $b\in B$.  We have the following definition as in
\cite[Definition 9.52]{EK2},
\cite[Section 4.2]{GHJ},
\cite[Definition 5.1.7]{JS}.
(This notion was introduced in \cite[Lemma 1.2.2]{P1}
originally.)

\begin{definition}{\rm
Let  
\[
\begin{array}{ccc}
B_{00} & \subset & B_{01}\\
\cap && \cap \\
B_{10} & \subset & B_{11}
\end{array}
\]
be inclusions of finite dimensional $C^*$-algebras
with faithful normalized trace $\tr$ on $B_{11}$.
When one of the following, mutually equivalent conditions
holds, we say that this is a \textsl{commuting square}.
\begin{enumerate}
\item We have $E_{B_{01}}=E_{B_{00}}$ on $B_{10}$.
\item We have $E_{B_{01}}(B_{10})=B_{00}$.
\item We have $E_{B_{10}}=E_{B_{00}}$ on $B_{01}$.
\item We have $E_{B_{10}}(B_{01})=B_{00}$.
\item We have $E_{B_{01}}E_{B_{10}}=E_{B_{00}}$.
\item We have $E_{B_{10}}E_{B_{01}}=E_{B_{00}}$.
\end{enumerate}
}\end{definition}

We now assume that all the four Bratteli diagrams for
$B_{00}\subset B_{01}$, $B_{00}\subset B_{10}$,
$B_{10}\subset B_{11}$, and $B_{01}\subset B_{11}$ are
connected.  We then have the following definition.

\begin{definition}\label{sym}{\rm
If we have ${\mathrm{span}\; }B_{01}B_{10}=B_{11}$ for
the above commuting square, we say that it is
\textsl{symmetric}.
}\end{definition}

See \cite[page 553]{EK2},
\cite[Definition 5.3.6]{JS},
\cite[Definition 1.8]{Sc} for other equivalent
conditions of being symmetric.  Also see
\cite[Corollary 5.3.4]{JS} for this equivalence.
As in \cite[Remark 5.3]{JS},
a symmetric commuting square is also called a
\textsl{non-degenerate} commuting square.
We now consider a symmetric commuting square
as in Definition \ref{sym}.

We also recall the following definition of the
basic construction \cite[Definition 3.1.1]{JS},
\cite[Definition 9.21]{EK2}.

\begin{definition}{\rm
Let $A\subset B$ be an inclusion of finite von Neumann
algebras with a normalized trace $\tr$ on $B$. 
Let $L^2(B)$ be the completion of $B$ with respect to
the inner product $\langle x,y\rangle=\tr(y^*x)$ for
$x,y\in B$ and identify $L^2(A)$ as the closure of $A$
within $L^2(B)$.  Let $B$ act on $L^2(B)$ by the left
multiplication.  Let $e_A$ be the orthogonal projection
from $L^2(B)$ onto $L^2(A)$.  We call $e_A$ the
\textsl{Jones projection}.  We call the von Neumann
algebra generated by $B$ and $e_A$ on $L^2(B)$ 
\textsl{basic construction}.  This basic construction
gives a finite von Neumann algebra with a natural trace.
}\end{definition}

We start with a commuting square as in Definition \ref{sym}.
We consider the basic construction $B_{02}$
for $B_{00}\subset B_{01}$ and the one $B_{12}$ for
$B_{10}\subset B_{11}$.  We can naturally identify the
Jones projection $e_{B_{00}}$ and $e_{B_{10}}$ and regard
$B_{02}$ as a subalgebra $B_{12}$.  Then 
\[
\begin{array}{ccc}
B_{01} & \subset & B_{02}\\
\cap && \cap \\
B_{11} & \subset & B_{12}
\end{array}
\]
is again a symmetric commuting square.  We can repeat this
procedure and obtain the following sequence of symmetric
commuting squares.
\[
\begin{array}{cccccccc}
B_{00} & \subset & B_{01}& \subset & B_{02}& \subset & B_{03}&\subset\cdots\\
\cap && \cap && \cap && \cap &\\
B_{10} & \subset & B_{11}& \subset & B_{12}& \subset & B_{13}&\subset\cdots.
\end{array}
\]

We can also apply the basic construction vertically to the original
commuting square.  The horizontal and vertical basic constructions
are compatible and we have
a double sequence $\{B_{kl}\}_{k,l}$ of finite
dimensional $C^*$-algebras with trace $\tr$.
We label $e_k$ for the vertical Jones projection for
$B_{k-1,0}\subset B_{k,0}$ and label $f_l$ for
the horizonal Jones  projection for $B_{0,l-1}\subset B_{0,l}$.
Note that the vertical Jones projection $e_k$ is for
the basic construction of
$B_{k-1,l}\subset B_{k,l}$ for all $l$ and
the horizonal Jones projection $f_l$
is for the basic construction of
$B_{k,l-1}\subset B_{k,l}$ for all $k$.
We now have the following relations.
\[
B_{kl}=\begin{cases}
B_{00},&\hbox{if\ }k=0\hbox{\ and\ }l=0,\\
\langle B_{10},e_1,e_2,\dots,e_{k-1}\rangle,
&\hbox{if\ }k>0\hbox{\ and\ }l=0,\\
\langle B_{01},f_1,f_2,\dots,f_{l-1}\rangle,
&\hbox{if\ }k=0\hbox{\ and\ }l>0,\\
\langle B_{11},e_1,e_2,\dots,e_{k-1},f_1,f_2,\dots,f_{l-1}\rangle,
&\hbox{if\ }k>0\hbox{\ and\ }l>0.\\
\end{cases}
\]
See \cite[Proposition 4.1.2]{GHJ},
\cite[Corollary 5.5.5]{JS} for more details.

We next take a nonzero projection $p\in B_{00}$ and
set $A_{kl}=pB_{kl}p$.  We normalize the trace on $A_{kl}$.  
Then the following is a commuting square again by
\cite[Proposition 4.2.6]{GHJ}.
\[
\begin{array}{ccc}
A_{k,l} & \subset & A_{k,l+1}\\
\cap && \cap \\
A_{k+1,l} & \subset & A_{k+1,l+1}.
\end{array}
\]
We then define $A_{k,\infty}$ to be the GNS-completion
of $\bigcup_{l=0}^\infty A_{k,l}$ with respect to the trace
$\tr$ and  $A_{\infty,l}$ to be the GNS-completion
of $\bigcup_{k=0}^\infty A_{k,l}$ with respect to the trace
$\tr$.  We also define $A_{\infty,\infty}$ to be the GNS-completion
of $\bigcup_{k=0}^\infty A_{k,k}$.  This is also the GNS-completions
of $\bigcup_{l=0}^\infty A_{\infty,l}$
and $\bigcup_{k=0}^\infty A_{k,\infty}$.
They are all hyperfinite II$_1$ factors, because
all the original Bratteli diagrams are finite and connected.
The subfactors $A_{k,\infty}\subset A_{k+1,\infty}$ and
$A_{\infty,l}\subset A_{\infty,l+1}$ have finite index.
The isomorphism classes of these subfactors are independent
of choice of $p$ by the relative McDuff property
\cite[Theorem 3.1]{B}.

Jones asked in 1995 whether the subfactor $A_{\infty,0}\subset A_{\infty,1}$
has a finite depth or not if so does 
$A_{0,\infty}\subset A_{1,\infty}$, and Sato gave a positive
answer to this question in \cite[Corollary 2.2]{S1}.
That is, the subfactor $A_{\infty,0}\subset A_{\infty,1}$
has a finite depth if and only if so does
$A_{0,\infty}\subset A_{1,\infty}$.  He further studied
this situation in \cite{S2}, \cite{S3}.

We reformulate this construction in terms of string algebras
and a bi-unitary connection as in \cite[Section 11.3]{EK2}.
Our symmetric commuting square 
\[
\begin{array}{ccc}
B_{00} & \subset & B_{01}\\
\cap && \cap \\
B_{10} & \subset & B_{11}
\end{array}
\]
gives a bi-unitary connection on four connected
graphs as in \cite[Theorem 11.2]{EK2},
\cite[Theorem 1.10]{Sc}.
(The notion of a bi-unitary connection was first introduced
in \cite{O1}, \cite{O2} with an extra axiom of flatness. 
See \cite[Definition 11.3]{EK2}
for the definition of a
bi-unitary connection.  It is also explained in
\cite[Definition 2.2]{K6}.)
Then our double sequence $\{A_{kl}\}_{k,l}$ is described
with the string algebra construction as in \cite[Section 11.3]{EK2}.
(If we have an increasing sequence of finite dimensional
$C^*$-algebras, then we have the Bratteli diagram.  Conversely,
if we have a Bratteli diagram given, we can construct the corresponding
increasing sequence of finite dimensional $C^*$-algebras canonically.
A pair of paths on the Bratteli diagram starting
at the top level and having the same ending vertex is called a
\textsl{string} and gives a canonical expression of a matrix
unit in a finite dimensional $C^*$-algebra.  The Jones projection
has a nice description in terms of string algebras as in
\cite[Definition 11.5]{EK2}.
See \cite[Definition 11.1]{EK2} for more details.)
However, there is a slight difference since we now have $A_{00}\neq\C$
in general. This matter is handled as follows.

Let $v_1,v_2,\dots,v_m$ be the
even vertices of ${\mathcal G}_0$ \cite[Section 10.3]{EK2}.
(The set of these vertices is denoted by $V_0$ in 
\cite[Section 2]{K6}.)  Let $A_{00}=\bigoplus_{j=1}^m M_{n_j}(\C)$,
where a direct summand $M_{n_j}(\C)$ corresponds to the vertex $v_j$.
(Here we have $n_j\ge0$ and at least one of $n_j$ is 
strictly positive.)  To deal with the matter of multiplicity,
we set $V$ be the set of $v_j^i$, where $1\le j\le m$ and
$1\le i\le n_j$.  If $v_j$ is connected to another vertex $w$ on
one of the four graphs, we interpret that $v_j^i$ is connected to
$w$ below.  Now our string is a pair $(\xi,\xi')$ of paths
on these graphs with $r(\xi)=r(\xi')$ as 
usual as in \cite[Section 11.3]{EK2}, but we now require
$s(\xi), s(\xi')\in V$ instead of $s(\xi)=s(\xi')$.
The multiplication and the $*$-operation are defined
in the same way as in \cite[Definition 11.4]{EK2}.
(This is a slight generalization of an idea of 
double starting vertices in \cite[Section 5]{K1}.)

Ocneanu's compactness argument \cite{O2} says that
we have $A'_{0,\infty}\cap A_{k,\infty}\subset A_{k,0}$
for the double sequence $\{A_{kl}\}_{k,l}$ as above.
A proof of the compactness argument is given in
\cite[Theorem 11.15]{EK2} only for the case $A_{00}=\C$, but
the same proof works for general $A_{00}$ without any change.

\section{Sato's construction and the main result}

We recall Sato's construction of a commuting square to realize
two mutually opposite Morita equivalent subfactors in \cite{S3}
in a slightly generalized form and
prove our main result saying that this construction covers
the fully general case.

We first recall some terminology related
to (opposite) Morita equivalence of
unitary fusion categories and subfactors.
We refer readers to \cite{EGNO} for general theory of fusion
categories and Morita equivalence.  
Also see \cite{K3} for a more brief description
relevant to subfactor theory (and conformal field theory).
A recent paper \cite{CJP} also explains the framework in
a concise manner.

Let $N\subset M$ be a type II$_1$ subfactor of finite index.
The $L^2$-completion $L^2(M)$ of $M$ with respect to the
trace $\tr$ gives an $N$-$M$ and $M$-$N$ bimodules
${}_N M_M$ and ${}_M M_N$.  (We simply write $M$ instead of
$L^2(M)$.)

We look at all irreducible $N$-$N$ bimodules arising from
${}_N M\otimes_N M\otimes_N\cdots\otimes_N M_N$.  If we have
only finitely many isomorphism classes for such irreducible
bimodules, then we say the subfactor has a \textsl{finite depth}.
This is equivalent to the condition that we have only finitely
many irreducible bimodules arising from irreducible decompositions
of ${}_M M\otimes_N M\otimes_N\cdots\otimes_N M_M$
up to isomorphism.  In this case, we consider a fusion
category of $N$-$N$ bimodules that are finite direct sums
of such irreducible $N$-$N$ bimodules.  We call it the fusion
category of $N$-$N$ bimodules arising from $N\subset M$.
We similarly define the fusion category of $M$-$M$ bimodules
arising from $N\subset M$.  These two fusion categories
are \textsl{Morita equivalent}.  We may realize any unitary fusion
category in this form using the hyperfinite II$_1$ factor
and any Morita equivalence between two such fusion categories
is realized in this way.
For a fusion category $\mathcal C$,
we define its global index to be the square sum of the 
(Perron-Frobenius) dimensions of all irreducible objects
up to isomorphism.  This is an invariant for Morita equivalence.
We define the global index of a subfactor
$N\subset M$ to be the global index of the fusion category
of $N$-$N$ bimodules arising from $N\subset M$.
(If a subfactor is of infinite depth, we define its global
index to be infinite.)

For a subfactor $N\subset M$, the opposite algebras, given by
reversing the order of the multiplication, give the opposite
subfactor $N^\opp\subset M^\opp$.  We say that the fusion 
category of the $N$-$N$ bimodules arising from $N\subset M$
and that of the $N^\opp$-$N^\opp$ bimodules 
arising from $N^\opp\subset M^\opp$ are \textsl{opposite Morita
equivalent}.
We also say that fusion categories ${\mathcal C}_1$
and ${\mathcal C}_2$ are opposite Morita equivalent, if
${\mathcal C}_1$ and ${\mathcal C}_3$ are Morita equivalent
and ${\mathcal C}_2$ and ${\mathcal C}_3$ are opposite Morita equivalent
for some fusion category ${\mathcal C}_3$.
(See \cite[Section 1]{S3} for its background.)

\begin{remark}{\rm
Sato simply said two fusion categories of bimodules are \textsl{equivalent}
in \cite{S2}, \cite{S3}
when we say they are Morita equivalent today.  His terminology
is also inconsistent with standard terminology of equivalence
of tensor categories, so we use ``Morita equivalence'' for this notion.
}\end{remark}

The following is essentially contained in \cite[Section 3]{S3}.

\begin{theorem}\label{sato}
Let $A\subset B$ and $C\subset D$ are Morita equivalent 
II$_1$ subfactors of finite index and finite depth and ${}_C X_A$ 
be a bimodule giving the Morita equivalence.
Define the finite dimensional $C^*$-algebras $A_{kl}$ as follows.
\[
A_{kl}=\begin{cases}
\End({}_C D\otimes_C D\otimes\cdots\otimes_C D\otimes_C
X \otimes_A B \otimes_A \cdots\otimes_A B_A),&\\
(k/2\hbox{\ copies of\ }D\hbox{\ and\ }
l/2\hbox{\ copies of\ }B),
&\hbox{if\ }k\hbox{\ and\ }l\hbox{\ are even},\\
\End({}_D D\otimes_C D\otimes\cdots\otimes_C D\otimes_C
X \otimes_A B \otimes_A \cdots\otimes_A B_A),&\\
((k+1)/2\hbox{\ copies of\ }D\hbox{\ and\ }
l/2\hbox{\ copies of\ }B),&\hbox{if\ }k\hbox{\ is\ odd\ and\ }
l\hbox{\ is even},\\
\End({}_C D\otimes_C D\otimes\cdots\otimes_C D\otimes_C
X \otimes_A B \otimes_A \cdots\otimes_A B_B),&\\
(k/2\hbox{\ copies of\ }D\hbox{\ and\ }
(l+1)/2\hbox{\ copies of\ }B),&\hbox{if\ }k\hbox{\ is\ even\ and\ }
l\hbox{\ is odd},\\
\End({}_D D\otimes_C D\otimes\cdots\otimes_C D\otimes_C
X \otimes_A B \otimes_A \cdots\otimes_A B_B),&\\
((k+1)/2\hbox{\ copies of\ }D\hbox{\ and\ }
(l+1)/2\hbox{\ copies of\ }B),&\hbox{if\ }k\hbox{\ and\ }
l\hbox{\ are odd}.\\
\end{cases}
\]
Then the double sequence $\{A_{kl}\}_{k,l}$ of commuting squares
gives a subfactor of finite index and finite depth.
\end{theorem}

The assumption on ${}_C X_A$ means that 
all the irreducible bimodules appearing in
${}_C X\otimes_A B \otimes_A\cdots\otimes_A B \otimes_A\bar X_C$
arise from the irreducible decompositions of
${}_C D\otimes_C\cdots\otimes_C D_C$.  In this case,
all the irreducible bimodules appearing in
${}_A \bar X\otimes_C D \otimes_C\cdots\otimes_C D \otimes_C X_A$
arise from the irreducible decompositions of
${}_A B\otimes_A\cdots\otimes_A B_A$ by the Frobenius reciprocity.
Note that we do not require irreducibility of ${}_C X_A$.

Another way to look at this ${}_C X_A$ is as follows.
We have a (possibly reducible)
$Q$-system (in the sense of \cite[Section 6]{L}) within 
the fusion category of the $A$-$A$ bimodules arising
from $A\subset B$ and it gives the that of the
$C$-$C$ bimodules arising from $C\subset D$ as the
dual fusion category.
(Longo \cite{L} uses type III factors and endomorphisms for
$Q$-systems.
See \cite[Proposition 2.1]{M} for the bimodule version
of the $Q$-system.)

After \cite[Remark 3.2]{S3}, Sato used a 
more restrictive form of ${}_C X_A$,
but his construction and proof work in the above setting.
Indeed, we need to check only connectedness of the four Bratteli
diagrams.  Suppose ${}_C Y_A$ and ${}_C Z_A$ are two irreducible
bimodules arising in some irreducible decomposition in the above
procedure.  Since ${}_A B_A$ generates the fusion category of
$A$-$A$ bimodules, we have
$\dim\Hom({}_C Y\otimes_A B\otimes_A\cdots\otimes_A B_A, {}_C Z_A)>0$
for sufficiently many copies of ${}_A B_A$.  This means that the
Bratteli diagram for $A_{2k,2l}\subset A_{2k,2l+1}$ is connected
for all $k, l$.  We can handle the other three Bratteli diagrams
similarly.

\begin{lemma}\label{intermediate}
Let $A\subset B\subset C\subset D$ be inclusions of type II$_1$ factors.
Suppose that $A\subset D$ is of finite index and finite depth and
that the global indices of $A\subset D$, $A\subset B$ and
$C\subset D$ are the same.  Then the bimodule ${}_B C_C$ gives
Morita equivalence between the fusion categories of
$B$-$B$ bimodules arising from $A\subset B$ and $C$-$C$ bimodules
arising from $C\subset D$.
\end{lemma}

\begin{proof}
Let ${\mathcal C}_A$ and ${\mathcal C}_D$ be the fusion categories
of $A$-$A$ bimodules and $D$-$D$ bimodules arising from $A\subset D$
respectively.  Let ${\mathcal C}_B$ be the fusion categories of
$B$-$B$ bimodules arising from the relative tensor products of
${}_B B_A$, $A$-$A$ bimodules in ${\mathcal C}_A$ and ${}_A B_B$.
Similarly, let ${\mathcal C}_C$ be the fusion categories of
$C$-$C$ bimodules arising from the relative tensor products of
${}_C D_D$, $D$-$D$ bimodules in ${\mathcal C}_D$ and ${}_D D_C$.
Then it is easy to see that the bimodule ${}_B C_C$ gives
Morita equivalence between ${\mathcal C}_B$ and ${\mathcal C}_C$.
Now by the assumption on the global indices, we see that
${\mathcal C}_B$ is given by $A\subset B$ and 
${\mathcal C}_C$ is given by $C\subset D$.
\end{proof}

We now give our main result.

\begin{theorem}\label{main}
A double sequence $\{A_{kl}\}_{k,l}$ of commuting squares
of finite dimensional $C^*$-algebras given in the form in Section 2
gives a subfactor
of finite index and finite depth if and only if it
is of the form of 
Theorem \ref{sato}, where $A,B,C,D$ are hyperfinite.

In this case, the subfactor
$A_{0,\infty}\subset A_{1,\infty}$ is anti-isomorphic to
$B\subset B_1$, where $B_1$ is given by the basic construction
of $A\subset B$.
\end{theorem}

\begin{proof}
The ``if'' part is already proved by Sato \cite{S3}, so we give a proof
for the converse.

Suppose that the subfactor $A_{0,\infty}\subset A_{1,\infty}$ has a 
finite depth.
Let $N=A_{0,\infty}$, $M=A_{1,\infty}$, $P=A_{\infty,0}$,
$Q=A_{\infty,1}$.
We define $A_{k,-1}=A'_{0,\infty}\cap A_{k,\infty}$ for $k\ge0$ 
and $A_{k,-2}=A'_{1,\infty}\cap A_{k,\infty}$ for $k\ge1$.
By the compactness argument, we have
$A_{k,-2}\subset A_{k,-1}\subset A_{k,0}$.
We define $R$ and $S$ to be the GNS-completions of
$\bigcup_{k=0}^\infty A_{k,-1}$ and 
$\bigcup_{k=1}^\infty A_{k,-2}$ with respect to the trace $\tr$,
respectively.  By the finite depth assumption on $N\subset M$,
these are also hyperfinite type II$_1$ factors.

As noted in \cite[page 371]{S1}, the following is a commuting
square for any $k\ge0$.
\[
\begin{array}{ccc}
A_{k,-1} & \subset & A_{k,0}\\
\cap && \cap \\
A_{k+1,-1} & \subset & A_{k+1,0}.
\end{array}
\]
For the same reason, the following also a commuting square
for any $k\ge1$.
\[
\begin{array}{ccc}
A_{k,-2} & \subset & A_{k,-1}\\
\cap && \cap \\
A_{k+1,-2} & \subset & A_{k+1,-1}.
\end{array}
\]

We apply the compactness argument to the following commuting
squares and the subfactor 
$R \subset A_{\infty,l}$ for some fixed $l$.
\[
\begin{array}{ccc}
A_{0,-1} & \subset & A_{0,l}\\
\cap && \cap \\
A_{1,-1} & \subset & A_{1,l}\\
\cap && \cap \\
A_{2,-1} & \subset & A_{2,l}\\
\downarrow && \downarrow \\
R & \subset & A_{\infty,l}.
\end{array}
\]
Then we have $R'\cap A_{\infty,l}\subset A_{0,l}$.
Since $A_{0,l}\subset N$, the converse inclusion is trivial, so
we have $R'\cap A_{\infty,l}=A_{0,l}$.
We similarly have $S'\cap A_{\infty,l}=A_{1,l}$.  That is,
the commuting square
\[
\begin{array}{ccc}
A_{k,l} & \subset & A_{k,l+1}\\
\cap && \cap \\
A_{k+1,l} & \subset & A_{k+1,l+1}
\end{array}
\]
is given as
\[
\begin{array}{ccc}
R'\cap A_{\infty,l} & \subset & R'\cap A_{\infty,l+1}\\
\cap && \cap \\
S'\cap A_{\infty,l} & \subset & S'\cap A_{\infty,l+1}.
\end{array}
\]
Using \cite[Lemma 11.1]{EK2}, we identify the above commuting with
the following commuting squares
\[
\begin{array}{ccc}
\End({}_R Q\otimes_P Q\otimes_P\cdots\otimes_P Q_P)
& \subset & 
\End({}_R Q\otimes_P Q\otimes_P\cdots\otimes_P Q\otimes_P Q_Q)
\\
\cap && \cap \\
\End({}_S Q\otimes_P Q\otimes_P\cdots\otimes_P Q_P)
& \subset & 
\End({}_S Q\otimes_P Q\otimes_P\cdots\otimes_P Q\otimes_P Q_Q), 
\end{array}
\]
where the number of copies of $Q$ is in the upper left corner is
$l/2$ if $l$ is even and
\[
\begin{array}{ccc}
\End({}_R Q\otimes_P Q\otimes_P\cdots\otimes_P Q_Q)
& \subset & 
\End({}_R Q\otimes_P Q\otimes_P\cdots\otimes_P Q_P)
\\
\cap && \cap \\
\End({}_S Q\otimes_P Q\otimes_P\cdots\otimes_P Q_Q)
& \subset & 
\End({}_S Q\otimes_P Q\otimes_P\cdots\otimes_P Q_P), 
\end{array}
\]
where the number of copies of $Q$ is in the upper left corenet is
$(l+1)/2$ if $l$ is odd.
We further have
$\End({}_R Q\otimes_P Q\otimes_P\cdots\otimes_P Q_P)
=\End({}_R P\otimes_P Q\otimes_P Q\otimes_P\cdots\otimes_P Q_P)$
and similar identities for the other three $\End$ spaces.
So our $A_{kl}$ is of the form of Theorem \ref{sato}, except
for the matter of Morita equivalence.

Since the finite depth assumption
implies strong amenability by \cite[Theorem 1]{P2},
the subfactor $S\subset R$ is anti-isomorphic 
to $N_1\subset N$ given by the downward basic construction
of $N\subset M$ by \cite[Theorem 2]{P2}.
This shows that the global indices of $N\subset M$ and $S\subset R$
are the same.  Since the global indices of $P\subset Q$ and
$N\subset M$ are the same by \cite[Corollary 2.5]{S1}, we know
that the global indices of $P\subset Q$ and $S\subset R$ are the
same.

By \cite[Theorem 2.4]{S2}, the fusion categories of $Q$-$Q$ bimodules
arising from $P\subset Q$ and $S\subset Q$ are the same, so the
global indices of $P\subset Q$ and $S\subset Q$ are the same.
By Lemma \ref{intermediate}, we now conclude that
our $A_{kl}$ is of the form of Theorem \ref{sato}. 

We have already seen in this case that $S\subset R$ is anti-isomorphic 
to $N_1\subset N$.
\end{proof}

\begin{remark}{\rm
In the above setting, we have $R=P$ if and only if the bi-unitary
connection is \textsl{flat}.  (See \cite[Definition 11.16]{EK2}.)
This is the case that was handled in 
\cite[Lemma 3.1]{S3}.
The canonical commuting square corresponds to the case
where $S=P_1$, $R=P$ and $P_1$ is a downward basic construction
of $P\subset Q$.
}\end{remark}

As a byproduct of the above characterization, we also
characterize a commuting square producing a given hyperfinite
II$_1$ subfactor of finite index and finite depth.
Let $N\subset M$ be a hyperfinite type II$_1$ subfactor 
of finite index and finite depth.
Let $P\subset Q$ be another hyperfinite
II$_1$ subfactor of finite index and finite depth that is
opposite Morita equivalent to $N\subset M$.  Let ${}_{N^\opp} X_P$ be a
$N^\opp$-$P$ bimodule giving Morita equivalence between the
fusion category of $N^\opp$-$N^\opp$ bimodules arising from 
$N_1^\opp\subset N^\opp$
and the one of $P$-$P$ bimodules arising from $P\subset Q$.
Using ${}_{N_1^\opp} N^\opp \otimes_{N^\opp} X \otimes_P Q_Q$ 
as a generator of a paragroup (replacing ${}_N M_M$ in
\cite[Chapter 10]{EK2}), 
we can realize the subfactor $N_1^\opp\subset Q$ and then we
also realize $N^\opp$ and $P$ as intermediate algebras, so
we have $N_1^\opp\subset N^\opp\subset P\subset Q$.
We now see that the above construction of commuting squares
$\{A_{0l}\subset A_{1l}\}_l$ gives a subfactor that is isomorphic
to $N\subset M$.  This gives the following.

\begin{corollary}\label{co1}
Any finite dimensional commuting squares
$\{A_{0,l}\subset A_{1,l}\}_l$ giving a hyperfinite II$_1$ subfactor
$A_{0,\infty}\subset A_{1,\infty}$
that is isomorphic to $N\subset M$ is of the above form.
\end{corollary}

Note that for a given fusion category $\mathcal C$, 
we have only finitely many
fusion categories that are Morita equivalent to $\mathcal C$.
(This is because we have only finitely many irreducible
$Q$-systems \cite[Section 6]{L} for a given fusion category. A 
$Q$-system is a unitary version of a Frobenius algebra.)
But after fixing the fusion category of $P$-$P$ bimodules,
we have countably many subfactors $P\subset Q$ and countably
many bimodules ${}_{M^\opp} X_P$.  (Note that we do not require
irreducibility of $P\subset Q$  or ${}_{M^\opp} X_P$
even when $N\subset M$ is
irreducible.)  So the total number of such commuting squares
for a fixed subfactor $N\subset M$ is countable.
Note it is trivial that we have countably many commuting squares
giving the same subfactor because if 
$\{A_{0,l}\subset A_{1,l}\}_l$  is one example, then
$\{A_{0,ml}\subset A_{1,ml}\}_l$ is another example for any $m$.

As in \cite{K6}, our bi-unitary connection is regarded as a
$4$-tensor related to $2$-dimensional topological order.  This theorem
also characterizes
which $4$-tensor realizes a given subfactor $N\subset M$.

\section{Examples}

In this section, we work out some examples to illustrate our results
in the previous section.

\begin{example}{\rm
Let $N\subset M$ be the Jones subfactor with principal graph $A_n$.
The fusion category of $N$-$N$ (or $M$-$M$) bimodules is isomorphic
to the even part of the Wess-Zumino-Witten category $SU(2)_{n-1}$.
All the fusion categories that are Morita equivalent to 
$SU(2)_{n-1}$ have been classified by Ocneanu \cite{O3}.
We now look at only the even part of these fusion categories.
(Also see \cite[Section 2]{KLPR} for classification of $Q$-systems.)
If $n=4m-3$, then we have a fusion category arising from the 
$\alpha$-induction for a simple current extension of order 2.
If $n=11$ or $n=29$, we also have
fusion categories arising from the $\alpha$-induction for conformal
embeddings $SU(2)_{10}\subset SO(5)_1$ or $SU(2)_{28}\subset (G_2)_1$.
These exhaust all
fusion categories that are Morita equivalent to the one arising
from a subfactor with principal graph $A_n$.
(See \cite{BEK1} for a general theory of $\alpha$-induction
and \cite[Section 5]{BEK2} for this classification.)
These give all possible commuting squares producing a subfactor
with principal graph $A_n$ as in Corollary \ref{co1}.
They are related to the Goodman-de la Harpe-Jones subfactors
\cite[Section 4.5]{GHJ}, \cite[Section 11.6]{EK2} as in \cite{G2}.

The easiest nontrivial example among these is the 
Goodman-de la-Harpe-Jones subfactor arising from $E_6$ having the
Jones index $3+\sqrt3$.  (See \cite[Proposition 4.5.2]{GHJ},
\cite[Example 11.25]{EK2}.)  This
subfactor is given as $A_{0,\infty}\subset A_{1,\infty}$,
and the subfactor $A_{\infty,0}\subset A_{\infty,1}$ has
principal graph $A_{11}$.  The graph ${\mathcal G}_1$ in
\cite[Fig. 11.65]{EK2} is not a principal graph of any subfactor,
but it is understood as in Theorem \ref{sato}.
}\end{example}

\begin{example}{\rm
If $n\neq 4m-3, 11$, then only fusion category that is Morita equivalent
to the $N$-$N$ bimodules of the subfactor with principal graph $A_n$
is itself.  In this case, the only commuting squares giving this
subfactor are the canonical commuting square, its basic construction
possibly cut with some projection in the higher relative commutants.
}\end{example}

\begin{example}{\rm
Let $A\subset B$ be a type II$_1$ subfactor with principal graph
$A_{4m-3}$, $C=B$ and $D$ be the crossed product of $B$ with 
the action of ${\mathbb Z}/2 {\mathbb Z}$ arising from the two
endpoints of the principal graph $A_{4m-3}$ as in 
$\alpha$ of \cite[Corollary 3.6]{G1}.  (We do not need 
hyperfiniteness of $B$ here, but if we do have hyperfiniteness,
then this action is of ${\mathbb Z}/2 {\mathbb Z}$
the ``orbifold action'' mentioned at the end of 
\cite[Section 3]{K2} arising from \cite{K1}.)
If we apply the construction in Theorem \ref{sato} to this
case, we get disconnected Bratteli diagrams for
$A_{2k,2l}\subset A_{2k,2l+1}$ for sufficiently large $k$ and
$l$, so the construction does not satisfy our requirement.
This failure is due to the fact the global index of $C\subset D$
is 2, which is smaller than that of $A\subset B$.
}\end{example}

In these examples so far, we have $P=R$ in Theorem \ref{main}.
We now look at an example where we have $P\neq R$.

\begin{example}{\rm
Consider a bi-unitary connection for the case all the four
graphs are the Dynkin diagram $E_7$ as in \cite[Fig. 11.32]{EK2}.
This bi-unitary connection is not flat and the principal
graph of the resulting subfactor is $D_{10}$ as shown in
\cite[Section 3]{EK1}.  In this case, all the three subfactors
$S\subset R$, $R\subset P$ and $P\subset Q$ have the principal
graph $D_{10}$.  (See \cite[Example 2.7]{S1} for the subfactor
$R\subset P$ in this example.)
}\end{example}

\begin{example}{\rm
Here is also another (rather trivial) example of $R\subsetneq P$.
Let $\{B_{kl}\}_{k,l}$ be the canonical commuting square of
a hyperfinite type II$_1$ subfactor of finite index and finite
depth.  Take a nonzero projection $p\in B_{2m,2n}$ and set
$A_{k,l}=p B_{2m+k,2n+l}p$, which gives commuting squares as above.
Then we see that $S\subset R$ and $P\subset Q$ are isomorphic
to the original subfactor and that $R\subset P$ is
isomorphic to $qP\subset qQ_{2m-1}q$ where $P\subset Q\subset Q_1
\subset\cdots$ is the Jones tower and $q$ is some nonzero
projection in $P'\cap Q_{2m-1}$.
In this example, the bi-unitary connection for $\{A_{kl}\}_{k,l}$
is the same as the original one for $\{B_{kl}\}_{k,l}$, but the
choice of the initial vertex $*$ for $A_{00}$ is different from
the canonical one.
}\end{example}

Also, fusion categories that are Morita equivalent to the $N$-$N$
bimodules of the Asaeda-Haagerup subfactor $N\subset M$ \cite{AH} are
also classified in \cite{GIS}.  We can, in principle, list 
commuting squares producing the Asaeda-Haagerup subfactor.


\begin{thebibliography}{99}  

\bibitem{AH} M. Asaeda and U. Haagerup,
Exotic subfactors of finite depth with Jones indices 
$(5+\sqrt{13)}/2$ and $(5+\sqrt{17})/2$,
\textit{Comm. Math. Phys.} {\bf 202} (1999), 1--63.

\bibitem{B}
D. Bisch, On the existence of central sequences in subfactors, 
\textit{Trans. Amer. Math. Soc.} {\bf 321} (1990), 117--128.

\bibitem{BKLR}
M. Bischoff, Y. Kawahigashi, R. Longo, K.-H. Rehren, 
``Tensor categories and endomorphisms of von Neumann algebras'' 
(with applications to Quantum Field Theory),
SpringerBriefs in Mathematical Physics {\bf 3}, (2015). 

\bibitem{BEK1}
J. B\"ockenhauer, D. E. Evans and Y. Kawahigashi, 
On $\alpha$-induction, chiral generators and modular 
invariants for subfactors,
\textit{Commun. Math. Phys.} {\bf 208} (1999), 429--487. 

\bibitem{BEK2}
J. B\"ockenhauer, D. E. Evans and Y. Kawahigashi, 
Chiral structure of modular invariants for subfactors,
\textit{Commun. Math. Phys.} {\bf 210} (2000), 733--784.

\bibitem{BMWSHV}
N. Bultinck, M. Mari\"en, D. J. Williamson, M. B. \c Sahino\u glu,
J. Haegeman and F. Verstraete,
Anyons and matrix product operator algebras,
\textit{Ann. Physics} {\bf 378} (2017), 183--233. 

\bibitem{CJP}
Q. Chen, C. Jones, D. Penneys,
A categorical Connes' $\chi(M)$,
preprint, arXiv:2111.06378.

\bibitem{EGNO} P. Etingof, S. Gelaki, D. Nikshych
and V. Ostrik, ``Tensor categories'', 
Mathematical Surveys and Monographs, {\bf 205},
American Mathematical Society, Providence (2015).

\bibitem{EK1} D. E. Evans and Y. Kawahigashi,
The $E_7$ commuting squares produce $D_{10}$ as principal graph,
\textit{Publ. Res. Inst. Math. Sci.} {\bf 30} (1994), 151--166. 

\bibitem{EK2} D. E. Evans and Y. Kawahigashi,
``Quantum Symmetries on Operator Algebras'',
Oxford University Press, Oxford (1998).

\bibitem{GHJ}
F. M. Goodman, P. de la Harpe, V. F. R. Jones,
``Coxeter graphs and towers of algebras'',
Mathematical Sciences Research Institute Publications, {\bf 14},
Springer-Verlag, New York (1989).

\bibitem{G1} S. Goto, 
Orbifold construction for non-AFD subfactors,
\textit{Internat. J. Math.} {\bf 5} (1994), 725--746.

\bibitem{G2} S. Goto, 
On the fusion algebras of bimodules arising from 
Goodman-de la Harpe-Jones subfactors,
\textit{J. Math. Sci. Univ. Tokyo} {\bf 19} (2012), 409--506.

\bibitem{GIS} P. Grossman, M. Izumi, N. Snyder, 
The Asaeda-Haagerup fusion categories,
\textit{J. Reine Angew. Math.} {\bf 743} (2018), 261--305.

\bibitem{J}
V. F. R. Jones, Index for subfactors, 
\textit{Invent. Math.} {\bf 72} (1983), 1--25.

\bibitem{JS} V. F. R. Jones, V. S. Sunder,
``Introduction to subfactors'', London Mathematical Society 
Lecture Note Series, {\bf 234}, Cambridge University Press, Cambridge
(1997).

\bibitem{K1}
Y. Kawahigashi, On flatness of Ocneanu's connections on the 
Dynkin diagrams and classification of subfactors, 
\textit{J. Funct. Anal.} {\bf 127} (1995), 63--107.

\bibitem{K2}
Y. Kawahigashi, Centrally trivial automorphisms and an analogue 
of Connes's $\chi(M)$ for subfactors,
\textit{Duke Math. J.} {\bf 71} (1993), 93--118.

\bibitem{K3}
Y. Kawahigashi,
Conformal field theory, tensor categories and operator algebras,
\textit{J. Phys. A} {\bf 48} (2015), 303001, 57 pp.

\bibitem{K4}
Y. Kawahigashi,
A remark on matrix product operator algebras, anyons and subfactors,
\textit{Lett. Math. Phys.} {\bf 110} (2020), 1113--1122.

\bibitem{K5}
Y. Kawahigashi, Two-dimensional topological order and operator algebras,
\textit{Internat. J. Modern Phys. B} 35 (2021), 2130003 (16 pages).

\bibitem{K6}
Y. Kawahigashi, Projector matrix product operators, anyons and
higher relative commutants of subfactors, preprint, arXiv:2102.04562.

\bibitem{KLPR}
Y. Kawahigashi, R. Longo, U. Pennig and K.-H. Rehren, 
Classification of non-local chiral CFT with $c<1$,
\textit{Comm. Math. Phys.} {\bf 271} (2007), 375--385.

\bibitem{LW} M. Levin and X.-G. Wen, String-net condensation: A
physical mechanism for topological phases, 
\textit{Phys. Rev. B} {\bf 71} (2005), 045110.

\bibitem{L} R. Longo,
A duality for Hopf algebras and for subfactors. I. 
\textit{Comm. Math. Phys.} {\bf 159} (1994), 133--150.

\bibitem{LFHSV}
L. Lootens, J. Fuchs, J. Haegeman, C. Schweigert and F. Verstraete,
Matrix product operator symmetries and intertwiners in string-nets 
with domain walls, 
\textit{SciPost Phys.} {\bf 10} (2021), 053.

\bibitem{M}
T. Masuda, 
An analogue of Longo's canonical endomorphism for 
bimodule theory and its application to asymptotic inclusions,
\textit{Internat. J. Math.} {\bf 8} (1997), 249--265. 

\bibitem{O1}
A. Ocneanu, 
Quantized groups, string algebras and Galois theory for algebras,
in: Operator algebras and applications, vol. 2, Warwick, 1987,
London Mathematical Society, Lecture Note Series, {\bf 136},
Cambridge University Press, Cambridge, (1988), pp. 119--172. 

\bibitem{O2}
A. Ocneanu,
``Quantum symmetry, differential geometry of finite graphs and
classification of subfactors'', University of Tokyo Seminary Notes 
{\bf 45}, (Notes recorded by Y. Kawahigashi), 1991.

\bibitem{O3}
A. Ocneanu, 
Paths on Coxeter diagrams: from Platonic solids and 
singularities to minimal models and subfactors, 
(Notes recorded by S. Goto), in
``Lectures on operator theory'', (ed. B. V. Rajarama Bhat et al.), 
The Fields Institute Monographs, Providence, (2000),
243--323.

\bibitem{P1}
S. Popa,
Maximal injective subalgebras in factors associated with free groups,
\textit{Adv. Math.} {\bf 50} (1983), 27--48.

\bibitem{P2}
S. Popa, Classification of amenable subfactors of type II,
\textit{Acta Math.} {\bf 172} (1994), 163--255.

\bibitem{S1}
N. Sato, 
Two subfactors arising from a non-degenerate commuting square.
An answer to a question raised by V. F. R. Jones,
\textit{Pacific J. Math.} {\bf 180} (1997), 369--376. 

\bibitem{S2}
N. Sato, 
Two subfactors arising from a non-degenerate commuting square. II,
Tensor categories and TQFT's. 
\textit{Internat. J. Math.} {\bf 8} (1997), 407--420.

\bibitem{S3}
N. Sato, 
Constructing a nondegenerate commuting square from 
equivalent systems of bimodules,
\textit{Internat. Math. Res. Notices} {\bf 1997}, no. 19, 967--981.

\bibitem{Sc}
J. K. Schou,
Commuting Squares and Index for Subfactors,
Ph.D. thesis, 1990, Odense University,
arXiv:1304.5907.

\end{thebibliography}
\end{document}